\documentclass[12pt]{amsart}
\usepackage{amsmath}
\usepackage{amssymb}
\usepackage[all]{xy}
\usepackage{longtable}

\newtheorem{theorem}{Theorem}[subsection]

\newtheorem{proposition}[theorem]{Proposition}

\theoremstyle{definition}

\theoremstyle{remark}
\newtheorem{remark}[theorem]{Remark}

\newcommand{\FF}{\mathbb{F}}
\newcommand{\ZZ}{\mathbb{Z}}
\newcommand{\QQ}{\mathbb{Q}}

\newcommand{\LL}{\mathbb{L}}
\newcommand{\TT}{\mathbb{T}}
\newcommand{\GG}{\mathbb{G}}
\newcommand{\EE}{\mathbb{E}}
\newcommand{\CC}{\mathbb{C}}
\newcommand{\PP}{\mathbb{P}}

\newcommand{\ba}{\mathbf{a}}
\newcommand{\be}{\mathbf{e}}
\newcommand{\bff}{\mathbf{f}}

\newcommand{\bm}{\mathbf{m}}

\newcommand{\bR}{\mathbf{R}}
\newcommand{\bS}{\mathbf{S}}

\newcommand{\cA}{\mathcal{A}}

\newcommand{\cM}{\mathcal{M}}

\newcommand{\cO}{\mathcal{O}}

\newcommand{\cT}{\mathcal{T}}

\newcommand{\rB}{\mathrm{B}}

\DeclareMathOperator{\GL}{GL}
\DeclareMathOperator{\Mat}{Mat}
\DeclareMathOperator{\Cent}{Cent}
\DeclareMathOperator{\End}{End}

\DeclareMathOperator{\Spec}{Spec}

\DeclareMathOperator{\Res}{Res}

\DeclareMathOperator{\Id}{Id}
\DeclareMathOperator{\im}{Im}
\DeclareMathOperator{\trdeg}{tr.deg}
\DeclareMathOperator{\wt}{wt}
\DeclareMathOperator{\divf}{div}

\newcommand{\ok}{\overline{k}}
\newcommand{\oU}{\overline{U}}
\newcommand{\oX}{\overline{X}}

\newcommand{\tpi}{\widetilde{\pi}}
\newcommand{\ttheta}{\widetilde{\theta}}
\newcommand{\tPsi}{\widetilde{\Psi}}
\newcommand{\oFqt}{\overline{\FF_q(t)}}

\newcommand{\Ga}{{\GG_{\mathrm{a}}}}
\newcommand{\Gm}{{\GG_{\mathrm{m}}}}

\newcommand{\smfrac}[2]{{\textstyle \frac{#1}{#2}}}

\begin{document}

\title[Geometric $\Gamma$-values and $\zeta$-values]{Geometric
Gamma values and zeta values in positive characteristic}

\author{Chieh-Yu Chang}
\address{National Center for Theoretical Sciences, Mathematics Division,
National Tsing Hua University, Hsinchu City 30042, Taiwan
  R.O.C.}
\address{Department of Mathematics, National Central
  University, Chung-Li 32054, Taiwan R.O.C.}

\email{cychang@math.cts.nthu.edu.tw}

\author{Matthew A. Papanikolas}
\address{Department of Mathematics, Texas A{\&}M University, College Station,
TX 77843-3368, USA} \email{map@math.tamu.edu}

\author{Jing Yu}
\address{Department of Mathematics, National Taiwan University, Taipei 106, Taiwan R.O.C. }
\email{yu@math.ntu.edu.tw}

\thanks{The first author was supported by a NCTS
postdoctoral fellowship. The second author was supported by NSF
Grant DMS-0600826.  The third author was supported by NSC Grant No.\
96-2628-M-007-006.}


\date{August 29, 2009}

\begin{abstract}
In analogy with values of the classical Euler $\Gamma$-function at
rational numbers and the Riemann $\zeta$-function at positive
integers, we consider Thakur's geometric $\Gamma$-function evaluated
at rational arguments and Carlitz $\zeta$-values at positive
integers.  We prove that, when considered together, all of the
algebraic relations among these special values arise from the
standard functional equations of the $\Gamma$-function and from the
Euler-Carlitz relations and Frobenius $p$-th power relations of the
$\zeta$-function.
\end{abstract}


\maketitle

\section{Introduction: a tale of two motives}
The period $\tpi$ of the Carlitz module is central to the world of
function field arithmetic.  Indeed it appears in several disparate
places, from explicit class field theory to Gauss sums to Drinfeld
modular forms (see \cite{Goss,Thakur}).  Moreover, $\tpi$ is closely
related to values of Thakur's geometric $\Gamma$-function
\cite{Thakur91} and the Carlitz $\zeta$-function
\cite{Carlitz1,Goss}.  Recent work of Anderson, Brownawell, and
Papanikolas \cite{ABP} (on $\Gamma$-values) and Chang and Yu
\cite{ChangYu} (on $\zeta$-values) have determined all algebraic
relations among special values of these functions individually.

Since $\tpi$ links these special values together, it is natural to
ask to what extent there are algebraic relations among
$\Gamma$-values and $\zeta$-values when considered together, and
this is the question addressed in the present paper.  The
answer anticipated by the theory of $t$-motives is that $\tpi$ is
the only link between $\Gamma$-values and $\zeta$-values, and indeed
our theorem (Theorem~\ref{T:MainThm}) shows that aside from
relations among $\Gamma$-values and $\zeta$-values involving $\tpi$
there are no other algebraic relations.  In summary this is the tale
of two motives, bringing the $t$-motive for $\Gamma$-values introduced in \cite{ABP} and
the $t$-motive for $\zeta$-values from \cite{ChangYu} on the same stage.

\subsection{Geometric $\Gamma$-values}
Let $\FF_q$ be the finite field with $q$ elements, where $q$ is a
power of a prime number $p$.  Let $A := \FF_q[\theta]$ and $k :=
\FF_q(\theta)$, where $\theta$ is a variable.  Let $A_+ \subseteq A$
be the subset of monic polynomials.  Let $k_\infty :=
\FF_q((1/\theta))$ be the completion of $k$ with respect to the
infinite absolute value on $k$, for which $|\theta|_\infty = q$. Let
$\CC_\infty$ be the completion of a fixed algebraic closure
$\overline{k_\infty}$ of $k_\infty$, and finally let $\ok$ be the
algebraic closure of $k$ in $\overline{k_\infty}$.

Working in analogy with the classical Euler $\Gamma$-function,
Thakur \cite{Thakur91} studied the geometric $\Gamma$-function over
$A$, which is a specialization of the two-variable $\Gamma$-function
of Goss~\cite{Goss88},
\[
  \Gamma(z) := \frac{1}{z} \prod_{n \in A_+} \biggl( 1 + \frac{z}{n} \biggr)^{-1}, \quad z \in \CC_\infty.
\]
It is a meromorphic function on $\CC_\infty$ with poles at zero and
$-n \in -A_+$ and satisfies several functional equations, which are
analogous to the translation, reflection, and Gauss multiplication
identities satisfied by the classical $\Gamma$-function.

\emph{Special Gamma values} are those values $\Gamma(r)$ with $r \in k \setminus A$.  Since, when $a \in A$, $\Gamma(a)$ is either $\infty$ or in $k$, we can restrict to special $\Gamma$-values for transcendence questions.  Now the functional equations for $\Gamma(z)$ induce families of algebraic relations among special $\Gamma$-values.  Moreover, if for $x$, $y \in \CC_\infty^\times$ we set $x \sim y$ when $x/y \in \ok^\times$, then for all $r \in k\setminus A$, $a \in A$, $g \in A_+$ with $\deg g = d$, we have the following relations:
\begin{itemize}
\item $\Gamma(r + a) \sim \Gamma(r)$,
\item$\prod_{\xi \in \FF_q^\times} \Gamma(\xi r) \sim \tpi$,
\item $\prod_{a \in A/g} \Gamma( \frac{r+a}{g} ) \sim \tpi^{\frac{q^d - 1}{q-1}}\Gamma(r)$.
\end{itemize}
Here $\tpi$ is algebraic over $k_\infty$ and is a fundamental period
for the Carlitz module, much in the same way that $2\pi\sqrt{-1}$ is
a fundamental period for the multiplicative group $\Gm$ over $\CC$.
In analogy with the transcendence of $2\pi\sqrt{-1}$ over $\QQ$,
Wade \cite{Wade} proved that $\tpi$ is transcendental over $k$.

As in the classical case, natural questions arise about the
transcendence and algebraic independence of these special
$\Gamma$-values, and much is now known.  As observed by Thakur
\cite{Thakur91}, for $q=2$ all values of $\Gamma(r)$, $r \in k
\setminus A$, are $\ok$-multiples of $\tpi$ and hence are
transcendental over $k$.  Thakur also related other special values
to periods of Drinfeld modules and deduced their transcendence.
Sinha \cite{Sinha} proved the first transcendence result for a
general class of special $\Gamma$-values: he showed that
$\Gamma(\frac{a}{f} + b)$ is transcendental over $k$ whenever $a$,
$f \in A_+$, $\deg a < \deg f$, and $b \in A$.  Sinha's result was
obtained by representing the $\Gamma$-values in question as periods
of certain $t$-modules over $\ok$ using the soliton functions of
Anderson \cite{Anderson92} and then invoking a transcendence
criterion of Gelfond-Schneider type established by Yu~\cite{Yu89}.
Expanding on Sinha's method, Brownawell and Papanikolas \cite{BP02}
represented all values $\Gamma(r)$, $r \in k \setminus A$, as
periods of $t$-modules over $\ok$ and thus proved transcendence for
all special $\Gamma$-values.

For algebraic relations among special $\Gamma$-values, Thakur
\cite{Thakur91} adapted the Deligne-Koblitz-Ogus criterion to this
setting and devised a diamond bracket criterion to determine which
algebraic relations among special $\Gamma$-values arise from the
functional equation relations.  More specifically, a
$\Gamma$-monomial is a monomial, with positive or negative
exponents, in $\tpi$ and special $\Gamma$-values, and Thakur's
criterion can decide whether a given $\Gamma$-monomial is in $\ok$.
In \cite{BP02}, Brownawell and Papanikolas showed that the only
$\ok$-linear relations among $1$, $\tpi$, and special
$\Gamma$-values are those explained by the diamond bracket
relations.  This result was obtained by analyzing the sub-$t$-module
structure of Sinha's $t$-modules and then invoking Yu's
sub-$t$-module theorem \cite{Yu97}, which plays the role here of
W\"ustholz's subgroup theorem~\cite{Wustholz89}.

In 2004, Anderson, Brownawell, and Papanikolas \cite{ABP} established
a new linear independence criterion (the so-called ``ABP-criterion''),
which is a motivic translation of Yu's sub-$t$-module theorem.  They
adapted Sinha's construction to create $t$-motives whose periods contain
the special $\Gamma$-values in question.  Again a key component was the
interpolation of special $\Gamma$-values via Anderson's soliton functions
\cite{Anderson92}.  Anderson, Brownawell, and Papanikolas used the
ABP-criterion to show that all algebraic relations over $\ok$ among
special $\Gamma$-values arise from diamond bracket relations among
$\Gamma$-monomials, and thus showed that all algebraic relations among
special $\Gamma$-values can be explained by the standard functional
equations.  As a consequence, the transcendence degrees of the fields
generated by special $\Gamma$-values can be obtained explicitly.

\begin{theorem}[Anderson-Brownawell-Papanikolas, {\cite[Cor.~1.2.2]{ABP}}] \label{T:ABPMain}
For all $f \in A_+$ of positive degree, the transcendence degree of
the field
\[
  \ok \bigl(\{ \tpi \} \cup \bigl\{ \Gamma(r) \bigm| r \in
  \smfrac{1}{f} A \setminus (\{0\} \cup -A_+) \bigr\} \bigr)
\]
over $\ok$ is $1 + \frac{q-2}{q-1}\cdot \#(A/f)^\times$.
\end{theorem}

\subsection{Carlitz $\zeta$-values}
There is also another set of special values closely related to the
Carlitz period $\tpi$.  In \cite{Carlitz1} Carlitz considered the
power sums
\[
  \zeta_C(n) := \sum_{a \in A_+} \frac{1}{a^n} \in k_\infty, \quad n = 1, 2, 3,\ldots,
\]
which are now called \emph{Carlitz zeta values.}  In analogy with
values of the Riemann $\zeta$-function at positive even integers,
Carlitz discovered that $\zeta_C(n)/\tpi^n$ is in $k$ whenever $n$
is divisible by $q-1$.  Thus we call a positive integer $n$
\textbf{even} if it is a multiple of $q-1$.  The ratios
$\zeta_C(n)/\tpi^n$ for \textbf{even} $n$ involve what are now
called Bernoulli-Carlitz numbers, and the theory is analogous to the
case of the Riemann $\zeta$-function. In particular, when $q=2$, all
$\zeta_C(n)$ are $k$-multiples of $\tpi^n$.

For each positive integer $n$, Anderson and Thakur
\cite{AndersonThakur} introduced the $n$-th tensor power $C^{\otimes
n}$ of the Carlitz module $C$ and explicitly related $\zeta_C(n)$ to
the last coordinate of the logarithm of a special algebraic point of
$C^{\otimes n}$.  Using this result, Yu \cite{Yu91} proved that each
$\zeta_C(n)$ is transcendental over $k$, and furthermore, when $n$
is not a multiple of $q-1$, he established the transcendence of
$\zeta_C(n)/\tpi^n$.  Later in 1997, Yu \cite{Yu97} proved that the
Euler-Carlitz relations are the only $\ok$-linear relations among
Carlitz $\zeta$-values at positive integers.

For algebraic relations among Carlitz $\zeta$-values, in addition to
the Euler-Carlitz relations, there are also the Frobenius $p$-th
power relations: for positive integers $m$, $n$,
\[
  \zeta_C(p^m n) = \zeta_C(n)^{p^m}.
\]
Chang and Yu \cite{ChangYu} extended Yu's previous results on
$\ok$-linear relations and proved that all algebraic relations over
$\ok$ among Carlitz $\zeta$-values at positive integers arise from
the Euler-Carlitz relations and the Frobenius relations.

\begin{theorem}[Chang-Yu {\cite[p.~322]{ChangYu}}] \label{T:ChangYu}
For any positive integer $s$, the transcendence degree of the field
\[
  \ok( \tpi, \zeta_C(1), \dots, \zeta_C(s))
\]
over $\ok$ is $s - \lfloor s/p \rfloor - \lfloor s/(q-1) \rfloor +
\lfloor s/(p(q-1)) \rfloor + 1$.
\end{theorem}

To prove this theorem Chang and Yu adapted methods of
Papanikolas~\cite{Papanikolas} on algebraic independence of Carlitz
logarithms to deduce algebraic independence results of Carlitz
polylogarithms.  In turn, they then proved the theorem by using
results of Anderson and Thakur \cite{AndersonThakur}, who showed
that Carlitz $\zeta$-values can be explicitly expressed in terms of
$k$-linear combinations of Carlitz polylogarithms with algebraic
arguments.

\subsection{Special $\Gamma$-values and Carlitz $\zeta$-values}
The main theorem of this paper is that, as one might expect, all the
algebraic relations among special $\Gamma$-values and Carlitz
$\zeta$-values arise from the standard relations.  One might anticipate
similar results in the classical setting, though little is known in
these directions to date.

\begin{theorem}[cf.\ Theorem~\ref{T:MainThmRedux}] \label{T:MainThm}
Given $f \in A_+$ with positive degree and $s$ a positive integer,
the transcendence degree of the field
\[
  \ok \bigl( \{\tpi\} \cup \bigl\{ \Gamma(r) \bigm| r \in
  \smfrac{1}{f} A \setminus (\{0\} \cup -A_+) \bigr\}
  \cup \{ \zeta_C(1), \dots, \zeta_C(s) \} \bigr)
\]
over $\ok$ is
\[
  1 + \frac{q-2}{q-1}\cdot \#(A/f)^\times + s - \lfloor s/p \rfloor - \lfloor s/(q-1) \rfloor + \lfloor s/(p(q-1)) \rfloor.
\]
\end{theorem}

Our tool for proving algebraic independence is the main theorem of
\cite{Papanikolas}, which is an application of the ABP-criterion
\cite{ABP} to monomials of periods and is in some sense a function
field version of Grothendieck's conjecture on periods of abelian
varieties.  The $t$-motive associated to special $\Gamma$-values
turns out to have geometric complex multiplication from a Carlitz
cyclotomic field, which enables us to show that the Galois group of
this ``$\Gamma$-motive'' is a torus contained inside a finite
product of the Weil restriction of scalars of $\Gm$ from the
cyclotomic field in question.  On the other hand, according to
\cite{ChangYu}, the Galois group of the ``$\zeta$-motive,'' i.e.\
the $t$-motive associated to Carlitz $\zeta$-values, is an extension
of $\Gm$ by a vector group.  Once we consider the direct sum of
these two $t$-motives, we deduce that the resulting Galois group is
an extension of a torus by a vector group and compute its dimension
to prove the result.

\begin{remark}
The $\Gamma$-function considered in this paper is ``geometric'' in
the sense that it is defined via the theory of Carlitz cyclotomic
covers of $\PP^1$ (see \S\ref{S:GammaMotives}).  In the theory of
function fields there is also an ``arithmetic'' $\Gamma$-function
studied by Carlitz and Goss \cite{Goss}, and which also has a rich
transcendence theory \cite{Thakur96}.  One may ask about algebraic
independence questions for the arithmetic $\Gamma$-function and its
relations with Carlitz $\zeta$-values as well.  These questions are
addressed in \cite{CPTY}.
\end{remark}

\section{Galois groups of $t$-motives}

\subsection{Notation and preliminaries}\

\begin{longtable}{p{0.5truein}@{\hspace{5pt}$=$\hspace{5pt}}p{5truein}}
$\FF_q$ & the finite field with $q$ elements, for $q$ a power of a prime number $p$. \\
$\theta$, $t$, $z$ & independent variables. \\
$A$ & $\FF_q[\theta]$, the polynomial ring in the variable $\theta$ over $\FF_q$. \\
$A_+$ & the set of monic elements of $A$.\\
$k$ & $\FF_q(\theta)$, the fraction field of $A$.\\
$k_\infty$ & $\FF_q((1/\theta))$, the completion of $k$ with respect to the place at infinity.\\
$\overline{k_\infty}$ & a fixed algebraic closure of $k_\infty$.\\
$\ok$ & the algebraic closure of $k$ in $\overline{k_\infty}$.\\
$\ttheta$ & a fixed choice in $\ok$ of a $(q-1)$-th root of $-\theta$.\\
$\CC_\infty$ & the completion of $\overline{k_\infty}$ with respect to the canonical extension of $\infty$.\\
$|\cdot|_\infty$ & the absolute value on $\CC_\infty$, normalized so that $|\theta|_\infty = q$.\\
$\TT$ & $\{ f \in \CC_\infty[[t]] \mid \textnormal{$f$ converges on $|t|_\infty \leq 1$} \}$, the Tate algebra over $\CC_\infty$.\\
$\LL$ & the fraction field of $\TT$.\\
$\Ga$ & the additive group. \\
$\GL_{r/F}$ & for a field $F$, the $F$-group scheme of invertible $r \times r$ matrices. \\
$\Gm$ & $\GL_1$, the multiplicative group.
\end{longtable}

For $n \in \ZZ$, given a Laurent series $f = \sum_i a_i t^i \in
\CC_\infty((t))$, we define the $n$-fold twist of $f$ by $f^{(n)} =
\sum_i a_i^{q^n} t^i$.  For each $n$, the twisting operation is an
automorphism of $\CC_\infty((t))$ and stabilizes several subrings,
e.g., $\ok[[t]]$, $\ok[t]$, and $\TT$.  More generally, for any
matrix $B$ with entries in $\CC_\infty((t))$, we define $B^{(n)}$ by
the rule ${B^{(n)}}_{ij} = B_{ij}^{(n)}$.  Also of some note (cf.\
\cite[Lem.~3.3.2]{Papanikolas}) is that
\[
  \FF_q[t] = \{ f \in \TT \mid f^{(-1)} = f \}, \quad \FF_q(t) = \{
  f \in \LL \mid f^{(-1)} = f \}.
\]
A power series $f = \sum_{i=0}^\infty a_i t^i \in \CC_\infty[[t]]$
that satisfies
\[
  \lim_{i\to \infty} \sqrt[i]{|a_i|_\infty} = 0, \quad [k_\infty(a_0,a_1,a_2, \ldots) : k_\infty] < \infty,
\]
is called an \emph{entire power series}.  As a function of $t$, such
a power series $f$ converges on all of $\CC_\infty$ and, when
restricted to $\overline{k_\infty}$, $f$ takes values in
$\overline{k_\infty}$.  The ring of entire power series is denoted
by $\EE$, and it is invariant under $n$-fold twisting $f \mapsto
f^{(n)}$.

\subsection{Galois groups and Frobenius difference equations} \label{S:GalGroups}
We follow \cite{Papanikolas} (see also \cite{Anderson86,ABP}) in
working with $t$-motives and their Galois groups.  Let
$\ok[t,\sigma]$ be the polynomial ring in variables $t$ and $\sigma$
subject to the relations,
\[
  ct=tc,\ \sigma t = t \sigma,\ \sigma c = c^{1/q} \sigma, \quad c \in \ok.
\]
Thus for $f \in \ok[t]$, one has $\sigma f = f^{(-1)} \sigma$.  An
\emph{Anderson $t$-motive} is a left $\ok[t,\sigma]$-module $\cM$,
which is free and finitely generated both as a left $\ok[t]$-module
and as a left $\ok[\sigma]$-module and which satisfies
\[
  (t-\theta)^N \cM \subseteq \sigma \cM,
\]
for integers $N$ sufficiently large.  The ring of Laurent
polynomials in $\sigma$ with coefficients in $\ok(t)$ is denoted
$\ok(t)[\sigma,\sigma^{-1}]$.  A \emph{pre-$t$-motive} is a left
$\ok(t)[\sigma,\sigma^{-1}]$-module that is finite dimensional over
$\ok(t)$.  The category of pre-$t$-motives is abelian.  Moreover,
there is a natural functor from the category of Anderson $t$-motives
to the category of pre-$t$-motives,
\[
  \cM \mapsto M := \ok(t) \otimes_{\ok[t]} \cM,
\]
where $\sigma$ acts diagonally on $M$.

In what follows we are interested in pre-$t$-motives $M$ that are
\emph{rigid analytically trivial.} To define this, we let $\bm \in
\Mat_{r\times 1}(M)$ be a $\ok(t)$-basis of $M$, and so
multiplication by $\sigma$ on $M$ is represented by
\[
  \sigma(\bm) = \Phi \bm,
\]
for some matrix $\Phi \in \GL_r(\ok(t))$.  Note that if $M$ arises
from an Anderson $t$-motive $\cM$ and $\bm \in \Mat_{r \times
1}(\cM)$, then necessarily also $\Phi \in \Mat_r(\ok[t])$.  Now $M$
is said to be rigid analytically trivial if there exists $\Psi \in
\GL_r(\LL)$ so that
\[
  \sigma(\Psi) = \Psi^{(-1)} = \Phi\Psi.
\]
The existence of such a matrix $\Psi$ is equivalent to the natural
map of $\LL$-vector spaces,
\[
  \LL \otimes_{\FF_q(t)} M^{\rB} \to M^\dagger,
\]
being an isomorphism, where
\begin{align*}
  M^\dagger &:= \LL \otimes_{\ok(t)} M, \textnormal{on which $\sigma$ acts diagonally},\\
  M^{\rB} &:= \textnormal{the $\FF_q(t)$-submodule of $M^\dagger$ fixed by $\sigma$, the `Betti cohomology' of $M$.}
\end{align*}
We then call $\Psi$ a rigid analytic trivialization for the matrix
$\Phi$, and in this situation $\Psi^{-1} \bm$ is an $\FF_q(t)$-basis
of $M^{\rB}$.

Rigid analytically trivial pre-$t$-motives form a neutral Tannakian
category over $\FF_q(t)$ with fiber functor $M \mapsto M^{\rB}$
\cite[Thm.~3.3.15]{Papanikolas}.  The strictly full Tannakian
subcategory generated by the images of Anderson $t$-motives is
called the category of \emph{$t$-motives}, which we denote by $\cT$.
By Tannakian duality, for each $t$-motive $M$ of dimension $r$ over
$\ok(t)$, the Tannakian subcategory $\cT_M$ generated by $M$ is
equivalent to the category of finite dimensional representations
over $\FF_q(t)$ of some algebraic group $\Gamma_M \subseteq
\GL_{r/\FF_q(t)}$.  This algebraic group is called the \emph{Galois
group} of the $t$-motive $M$.  Note that in this situation we always
have a faithful representation
\[
  \varphi : \Gamma_M \hookrightarrow \GL(M^{\rB}),
\]
which is called the tautological representation of $M$.  The
following theorem connects Galois groups of $t$-motives to the
transcendence degrees of interest.

\begin{theorem}[Papanikolas {\cite[Thm.~1.1.7]{Papanikolas}}] \label{T:PapanikolasMain}
Let $M$ be a $t$-motive with Galois group $\Gamma_M$.  Suppose that
$\Phi \in \GL_r(\ok(t)) \cap \Mat_r(\ok[t])$ represents
multiplication by $\sigma$ on $M$ and that $\det \Phi =
c(t-\theta)^s$, $c \in \ok^\times$.  Let $\Psi \in \GL_r(\TT) \cap
\Mat_r(\EE)$ be a rigid analytic trivialization of $\Phi$.  Then the
transcendence degree of the field
\[
  \ok\bigl( \Psi_{ij}(\theta) \mid 1 \leq i, j \leq r \bigr)
\]
over $\ok$ is equal to the dimension of $\Gamma_M$.
\end{theorem}

Papanikolas \cite{Papanikolas} has further connected these Galois
groups to the Galois theory of Frobenius difference equations, in
analogy with classical differential Galois theory.  This provides a
method for the explicit computation of $\Gamma_M$ for a $t$-motive
$M$. Assume that $\Phi$ and $\Psi$ are chosen for $M$ as above, and
let $\Psi_1$, $\Psi_2 \in \GL_r(\LL \otimes_{\ok(t)} \LL)$ be the
matrices such that $(\Psi_1)_{ij} = \Psi_{ij} \otimes 1$ and
$(\Psi_2)_{ij} = 1 \otimes \Psi_{ij}$.  Let $\tPsi := \Psi_1^{-1}
\Psi_2 \in \GL_r(\LL \otimes_{\ok(t)} \LL)$.  Now let $X :=
(X_{ij})$ be an $r \times r$ matrix of independent variables.  We
define an $\FF_q(t)$-algebra homomorphism
\[
  \mu_\Psi : \FF_q(t)[X,1/\det X] \to \LL \otimes_{\ok(t)} \LL,
\]
by setting $\mu_\Psi(X_{ij}) = \tPsi_{ij}$.  Finally we set
\begin{equation} \label{E:GammaPsiDefn}
  \Gamma_\Psi := \Spec \im \mu_\Psi.
\end{equation}
Letting $\Lambda_\Psi$ be the field generated over $\ok(t)$ by the
entries of $\Psi$, we obtain the following results.

\begin{theorem}[Papanikolas {\cite[Thms.~4.2.11, 4.3.1, 4.5.10]{Papanikolas}}]
\label{T:PapanikolasGalGrp}
Let $M$ be a $t$-motive with Galois group $\Gamma_M$.  Let $\Phi \in
\GL_r(\ok(t))$ represent multiplication by $\sigma$ on $M$, and let
$\Psi$ be a rigid analytic trivialization for $\Phi$.  Then
\begin{enumerate}
\item[(a)] $\Gamma_\Psi$ is an affine algebraic group scheme over $\FF_q(t)$,
\item[(b)] $\Gamma_\Psi$ is absolutely irreducible and smooth over $\oFqt$,
\item[(c)] $\dim \Gamma_\Psi = \trdeg_{\ok(t)} \Lambda_\Psi$,
\item[(d)] $\Gamma_\Psi$ is isomorphic to $\Gamma_M$ over $\FF_q(t)$.
\end{enumerate}
\end{theorem}

\begin{remark} \label{R:TautRep}
Using the formalism of Frobenius difference equations, the
tautological representation of $M$ can be described as follows.  For
any $\FF_q(t)$-algebra $R$, the map $\varphi : \Gamma_M(R) \to \GL(R
\otimes_{\FF_q(t)} M^{\rB})$ is given by
\[
  \gamma \mapsto \bigl(1 \otimes \Psi^{-1}\bm \mapsto
  (\gamma^{-1} \otimes 1) \cdot (1 \otimes \Psi^{-1} \bm) \bigr).
\]
\end{remark}

\subsection{Carlitz theory and $t$-motives}
Put $D_0 := 1$ and set
\[
D_i = \prod_{j=0}^{i-1} \bigl( \theta^{q^i} - \theta^{q^j} \bigr),
\quad i = 1, 2, \ldots.
\]
The Carlitz exponential is defined by the power series
\[
  \exp_C(z) := \sum_{i=0}^\infty \frac{z^{q^i}}{D_i},
\]
and as a function on $\CC_\infty$ it is an entire function and
satisfies the functional equation
\[
  \exp_C(\theta z) = \theta\exp_C(z) + \exp_C(z)^q.
\]
Moreover, one has the product expansion
\[
  \exp_C(z) = z \prod_{0 \neq a \in A} \biggl( 1 - \frac{z}{a\tpi} \biggr),
\]
where
\[
  \tpi = \theta\ttheta\, \prod_{i=1}^\infty \Bigl( 1 - \theta^{1-q^i} \Bigr)^{-1},
\]
is the fundamental period of Carlitz \cite[Ch.~3]{Goss}.  Set $L_0
:= 1$ and
\[
L_i := \prod_{j=1}^{i} \bigl( \theta-\theta^{q^j} \bigr), \quad i =
1, 2, \ldots.
\]
The Carlitz logarithm is defined by the power series
\[
  \log_C(z) := \sum_{i=0}^\infty \frac{z^{q^i}}{L_i},
\]
which converges $\infty$-adically for all $z \in \CC_\infty$ with
$|z|_\infty < |\theta|_\infty^{\frac{q}{q-1}}$.  It satisfies the
functional equation
\[
  \theta\log_C(z) = \log_C(\theta z) + \log_C(z^q),
\]
whenever the values in question are defined.  As formal power series
$\exp_C(z)$ and $\log_C(z)$ are inverses of each other with respect
to composition.

For a positive integer $n$, the $n$-th Carlitz polylogarithm is
defined by the series
\begin{equation} \label{E:CarlitzPolyLog}
  \log_C^{[n]}(z) := \sum_{i=0}^\infty \frac{z^{q^i}}{L_i^n},
\end{equation}
which converges $\infty$-adically for all $z \in \CC_\infty$ with
$|z|_\infty < |\theta|_\infty^{\frac{nq}{q-1}}$.  It can be checked
that $\log_C^{[n]}(z)$ is injective on its domain of convergence.

To connect the Carlitz theory to $t$-motives, we consider
\[
  \Omega(t) := \ttheta^{-q} \prod_{i=1}^\infty \biggl( 1 - \frac{t}{\theta^{q^i}} \biggr) \in \EE.
\]
It satisfies the functional equation
\begin{equation} \label{E:OmegaFnEq}
  \Omega^{(-1)} = (t-\theta)\Omega
\end{equation}
and $\Omega(\theta) = -1/\tpi$ (see \cite[\S 5.1]{ABP} for details).
The Carlitz motive $C$ is the $t$-motive with underlying vector
space $\ok(t)$ itself and with $\sigma$-action given by
\[
  \sigma g = (t-\theta)g^{(-1)}, \quad g \in \ok(t).
\]
Thus in this case $r=1$ and $\Phi = t-\theta$, and $\Omega$ provides a
rigid analytic trivialization by \eqref{E:OmegaFnEq}.  For $n\geq
1$, the $n$-th tensor power $C^{\otimes n} := C \otimes_{\ok(t)}
\cdots \otimes_{\ok(t)} C$ of the Carlitz motive also has $\ok(t)$
as its underlying space, but with $\sigma$-action $\sigma g =
(t-\theta)^n g^{(-1)}$, for $g \in \ok(t)$.  Thus $\Omega^n$ is a
rigid analytic trivialization of $C^{\otimes n}$.  By
Theorem~\ref{T:PapanikolasGalGrp}, it follows that the Galois group
of $C^{\otimes n}$ is $\Gm$ over $\FF_q(t)$, since $\Omega$ is
transcendental over $\ok(t)$.

Finally we note that the Carlitz polylogarithms are also
specialization at $t=\theta$ of elements from $\TT$.  To see this, fix a positive
integer $n$, and choose any $\alpha \in \ok^{\times}$ with
$|\alpha|_\infty < |\theta|_\infty^{\frac{nq}{q-1}}$.  We consider
the power series
\begin{equation}\label{E:Lalpha}
  L_{\alpha,n}(t) := \alpha + \sum_{i=1}^\infty \frac{\alpha^{q^i}}{(t-\theta^q)^n (t-\theta^{q^2})^n \cdots (t-\theta^{q^i})^n} \quad \in \TT.
\end{equation}
Specializing at $t=\theta$, we see that $L_{\alpha,n}(\theta)$ is
exactly the polylogarithm $\log_C^{[n]}(\alpha)$ by
\eqref{E:CarlitzPolyLog}.  We observe that $L_{\alpha,n}$ satisfies
the functional equation,
\begin{equation} \label{E:LalphanFnEq}
  (\Omega^n L_{\alpha,n})^{(-1)} = \alpha^{(-1)}(t-\theta)^n \Omega^n + \Omega^n L_{\alpha,n},
\end{equation}
which is key for defining the the $\zeta$-motives in
\S\ref{S:ZetaMotive}.

\section{Special $\Gamma$-values and $t$-motives} \label{S:GammaMotives}
Throughout this section we fix $f \in A_+$ with positive degree, and
we let $\ell$ be the cardinality of $(A/f)^\times$.  We review here
information about the $t$-motives associated to special
$\Gamma$-values from \cite[\S 5--6]{ABP}, and at the end we describe
their Galois groups.  We apologize in advance for the confluence of
notation in $\Gamma$, $\Gamma_M$, and $\Gamma_\Psi$, so throughout
we attempt to be as clear as possible as to which ``$\Gamma$'' we
are using.

\subsection{Carlitz cyclotomic covers}  For $x \in k_\infty$, we
let $\be(x) := \exp_C(\tpi x)$, and as such $\be : k_\infty \to
k_\infty(\ttheta)$.  As is well-known, $k(\be(1/f))$ is a Galois
extension of $k$ with Galois group $(A/f)^\times$, and it is the
Carlitz analogue of a cyclotomic extension of $\QQ$.  The Galois
action of $(A/f)^\times$ on $k(\be(1/f))$ is induced by the Artin
automorphism; moreover, for $a \in A$ relatively prime to $f$, the
action of $a$ on $k(\be(1/f))$ is defined by
\[
  a :\be(1/f) \mapsto \be(a/f).
\]

Now given $a \in A$, write $a = \theta b + \epsilon$, with $b \in A$
and $\epsilon \in \FF_q$, and define recursively a polynomial
$C_a(t,z) \in \FF_q[t,z]$ by
\[
C_a(t,z) := \begin{cases}
0 & \textnormal{if $a=0$,} \\
C_b(t,tz+z^q)+\epsilon z & \textnormal{if $a \neq 0$.}
\end{cases}
\]
The polynomial $C_a(t,z)$ is called a division polynomial of $\be$,
and it has the following properties:
\[
  C_a(\theta,\be(x)) = \be(ax),\ C_a(t,C_b(t,z)) = C_{ab}(t,z), \quad
  \forall\, a, b \in A, x \in k_\infty.
\]
Furthermore,
\[
  C_f(\theta,z) = \prod_{\substack{a \in A \\ \deg a < \deg f}} \bigl( z - \be (a/f ) \bigr).
\]
Hence there is a unique factor $C_f^\star(t,z) \in \FF_q[t,z]$ of
$C_f(t,z)$ so that
\[
  C_f^\star(\theta,z) = \prod_{\substack{a \in A \\ \deg a < \deg f,\ (a,f)=1}}
  \bigl( z - \be (a/f) \bigr),
\]
and we can consider $C_f^\star(t,z)$ to be the $f$-th cyclotomic
polynomial.  One finds that $C_f(t,z)$ is irreducible in
$\FF_q[t,z]$ and remains so in $\ok[t,z]$; and that the rings $\bR_f
:= \FF_q[t,z]/(C_f^\star(t,z))$ and $\bS_f :=
\ok[t,z]/(C_f^\star(t,z))$ are Dedekind domains.

Let $U_f/\FF_q$ be the affine curve whose coordinate ring is
$\bR_f$, and let $X_f/\FF_q$ be its non-singular projective model.
The infinite points of $X_f$ are defined to be the closed points in
$X_f \setminus U_f$, and they are all $\FF_q$-rational.  For $a \in
A$ relatively prime to $f$, set
\[
  \xi_a := \bigl( \theta, \be(a/f) \bigr),
\]
and note that the collection $\{ \xi_a \}$ is the collection of
$\ok$-points of $U_f$ above the point $t=\theta$ on the affine
$t$-line. Let
\[
  \oU_f := \ok \times_{\FF_q} U_f, \quad \oX_f:= \ok \times_{\FF_q} X_f,
\]
be the scalar extensions to $\ok$.  Since $C_f^\star(t,z)$ is
absolutely irreducible, we see that $X_f(\ok) = \oX_f(\ok)$ and that
$U_f(\ok) = \oU_f(\ok)$.  For $n \in \ZZ$, the $n$-fold twisting
operation extends in a natural way to $\bS_f$ that leaves $\bR_f$
fixed.  This action induces an action $x \mapsto x^{(n)}$ on
$\oX_f(\ok)$, which raises the coordinates of $x$ to the $q^n$-th
power.

\subsection{Geometric constructions for $t$-motives}
Let $\cA_f$ be the free abelian group on symbols of the form $[x]$,
where $x \in \smfrac{1}{f}A/A$.  Every $\ba \in \cA_f$ has a unique
expression of the form
\[
  \ba = \sum_{\substack{a \in A \\ \deg a < \deg f}} m_a [ a/f ],
  \quad m_a \in \ZZ.
\]
If all of the coefficients $m_a$ are non-negative, then we say that
$\ba$ is effective.  Let $\wt : \cA_f \to \ZZ[1/(q-1)]$ be the
unique group homomorphism such that for $x \in \frac{1}{f}A/A$,
\[
  \wt [x] = \begin{cases}
  0 & \textnormal{if $x \in A$,} \\
  \frac{1}{q-1} & \textnormal{if $x \notin A$.}
  \end{cases}
\]
For each $a \in A$ relatively prime to $f$, there exists a unique
automorphism $(\ba \mapsto a \star \ba) : \cA_f \to \cA_f$ of
abelian groups such that
\[
  a \star [x] = [ax], \quad x \in \smfrac{1}{f}A/A.
\]
Hence $(A/f)^\times$ acts on $\cA_f$ via $\star$.  Finally we define
\[
  \Pi(z) := z\Gamma(z) = \prod_{a \in A_+} \biggl( 1 + \frac{z}{a} \biggr)^{-1},
\]
which is sometimes called the ``geometric factorial'' function, and
for $\ba \in \cA_f$ we define $\Pi(\ba) \in \CC_\infty^\times$ so
that
\[
\Pi([x]) = \Pi(x), \quad x \in \smfrac{1}{f}A/A,\ |x|_\infty < 1.
\]
The elements of the image of $\Pi$ on $\cA_f$ are called
$\Pi$-monomials of level $f$.

For all $x \in k_\infty$ and integers $N \geq 0$, we define the
\emph{diamond bracket} by setting,
\[
  \langle x \rangle_N := \begin{cases}
  1 & \textnormal{if $\inf_{a\in A} |x-a-\theta^{-N-1}|_\infty < |\theta|_\infty^{-N-1}$,} \\
  0 & \textnormal{otherwise,}
  \end{cases}
\]
and
\[
  \langle x \rangle := \sum_{N=0}^\infty \langle x \rangle_N.
\]
The sum on the right has at most one non-zero term, and in
particular converges, and the value $\langle x \rangle$ is either
$0$ or $1$.  We can extend $\langle \cdot \rangle_N$ to $\cA_f$ by
setting
\[
  \langle [x] \rangle_N = \langle x \rangle_N, \quad x \in \smfrac{1}{f}A/A,
\]
and set
\[
  \langle \ba \rangle  = \sum_{N=0}^\infty \langle \ba \rangle_N, \quad \ba \in \cA_f.
\]

Fix an effective $\ba \in \cA_f$, with $\wt\ba > 0$.  We define
effective divisors of $\oX_f$ by the formulas,
\[
  \xi_\ba := \sum_{\substack{a \in A \\ \deg a < \deg f,\ (a,f)=1}}
  \langle a \star \ba \rangle \cdot \xi_a, \quad
  W_\ba := \sum_{\substack{a \in A \\ \deg a < \deg f,\ (a,f) = 1}}
  \sum_{N=1}^\infty \langle a \star \ba \rangle_N \cdot \sum_{i=0}^{N-1} \xi_a^{(i)}.
\]
Let $\infty_{X_f}$ be the divisor sum of the infinite points of
$\oX_f$ multiplied by $q-1$. By \cite[\S 6.3.9]{ABP}, one can use Anderson's solitons \cite{Anderson92} to define
a function $g_\ba$ on $\oX_f$, which is a rational
function that is regular on $\oU_f$ and has divisor
\[
  \divf(g_\ba) = -(\wt \ba) \cdot \infty_{X_f} + \xi_\ba + W_\ba^{(1)} - W_\ba.
\]
These functions provide geometric interpolations of
$\Pi$-monomials by way of the infinite product,
\[
  \Pi(a \star \ba)^{-1} = \prod_{N=1}^\infty g_\ba^{(N)} (\xi_a),
\]
which converges in $\CC_\infty$.

Now let $\tilde{H}(\ba)$ be the left $\ok[t,\sigma]$-module, whose
underlying $\ok[t]$-module is $\bS_f$ with $\sigma$-action given by
\[
  \sigma h = g_\ba h^{(-1)}, \quad h \in \tilde{H}(\ba).
\]
Define
\[
  H(\ba) := H^0(\oU_f,\cO_{\oX_f} (-W_\ba^{(1)})) \subseteq H^0(\oU_f,\cO_{\oX_f}) = \tilde{H}(\ba),
\]
which is, as an ideal of $\bS_f$, projective of rank one over
$\bS_f$ and free of rank $\ell$ over $\ok[t]$.  The module $H(\ba)$
is an Anderson $t$-motive by \cite[\S 6.4.2]{ABP}.  Now since
$\bR_f$ commutes with $\sigma$, it follows that there is a natural
inclusion
\[
  \bR_f \subseteq \End_{\ok[t,\sigma]}(H(\ba)).
\]
The main theorem of \cite{ABP} that describes the essential
properties of $H(\ba)$ and its connections to special
$\Gamma$-values is the following.

\begin{theorem}[Anderson-Brownawell-Papanikolas {\cite[Prop.~6.4.4]{ABP}}] \label{T:ABPPhiPsiPi}
Let $\ba \in \cA_f$ be effective with $\wt \ba > 0$, and let
$H(\ba)$ be the Anderson $t$-motive defined above.  Suppose that
$\Phi_\ba \in \Mat_{\ell}(\ok[t])$ represents multiplication by
$\sigma$ on $H(\ba)$ with respect to any $\ok[t]$-basis.  Then there
exists a rigid analytic trivialization $\Psi_\ba \in \GL_{\ell}(\TT)
\cap \Mat_{\ell}(\EE)$ of $\Phi_\ba$ with the property that the sets
\[
 \{ \Psi_\ba(\theta)_{ij} \mid i, j = 1,\dots, \ell \}, \quad
 \{\Pi(a \star \ba)^{-1} \mid a \in A,\ (a,f) = 1 \},
\]
span the same $\ok$-subspace of $\overline{k_\infty}$.  In
particular, $H(\ba)$ is a rigid analytically trivial Anderson
$t$-motive, and a transcendence basis of $\ok(\Psi_\ba(\theta)_{ij}
\mid i, j = 1, \dots, \ell)$ over $\ok$ can be chosen among elements
of the set of $\Pi$-monomials $\{ \Pi(a \star \ba)^{-1} \mid a \in
A,\ (a,f) = 1 \}$ that are linearly independent over $\ok$.
\end{theorem}

\subsection{The $\Gamma$-motive $M_f$} \label{S:GammaMotive}
For an effective $\ba \in \cA_f$ with $\wt \ba > 0$, we define
$M_\ba := \ok(t) \otimes_{\ok[t]} H(\ba)$, which is a $t$-motive in
the sense of \S\ref{S:GalGroups}.

\begin{proposition} \label{P:GammaMaTorus}
The Galois group $\Gamma_{M_\ba}$ is a torus over $\FF_q(t)$.
\end{proposition}

\begin{proof}
Let $\Phi_\ba \in \Mat_{\ell}(\ok[t])$ represent multiplication by
$\sigma$ on $M_\ba$ with respect to a $\ok(t)$-basis $\bm \in
\Mat_{\ell \times 1}(M_\ba)$ of $M_\ba$, which is the scalar
extension of a fixed $\ok[t]$-basis of $H(\ba)$. Let $\Psi_\ba \in
\GL_\ell(\TT) \cap \Mat_\ell(\EE)$ be a rigid analytic
trivialization of $\Phi_\ba$ as in Theorem~\ref{T:ABPPhiPsiPi}.
Finally, let $\End(M_\ba)$ denote the ring of endomorphisms of
$M_\ba$ as a left $\ok(t)[\sigma,\sigma^{-1}]$-module.

For $\bff \in \End(M_\ba)$ there is a matrix $F \in
\Mat_\ell(\ok(t))$ so that $\bff(\bm) = F\bm$.  Since $\bff \sigma =
\sigma \bff$, we have $\Phi_\ba F = F^{(-1)}\Phi_\ba$, from which it
follows that the matrix $\Psi_\ba^{-1}F\Psi_\ba \in \Mat_\ell(\LL)$
is fixed by $\sigma$ and hence $\Psi_\ba^{-1}F\Psi_\ba \in
\Mat_\ell(\FF_q(t))$.  Therefore, we have an injective map,
\[
  \bigl(\bff \mapsto \bff^{\rB} := \Psi_\ba^{-1} F \Psi_\ba\bigr) :
  \End(M_\ba) \hookrightarrow \End_{\FF_q(t)}(M_\ba^{\rB}) = \Mat_\ell(\FF_q(t)).
\]
Since the tautological representation $\varphi : \Gamma_{M_\ba}
\hookrightarrow \GL(M_\ba^{\rB})$ is functorial in $M_\ba$ (cf.\
\cite[Thm.~4.5.3]{Papanikolas}), it follows that for an
$\FF_q(t)$-algebra $R$, any $\gamma \in \Gamma_{M_\ba}(R)$, and any
$\bff \in \End(M_{\ba})$, we have the following commutative diagram.
\[
\xymatrix{R \otimes_{\FF_q(t)}M_\ba^{B} \ar[r]^{\varphi(\gamma)}
 \ar@{->}[d]^{1\otimes \bff^{B}}
 & R\otimes_{\FF_q(t)}M_\ba^{B} \ar@{->}[d]^{1\otimes \bff^{B}}\\
 R\otimes_{\FF_q(t)}M_\ba^{B}
 \ar[r]^{\varphi(\gamma)}& R\otimes_{\FF_q(t)}M_\ba^B.}
\]
Therefore, we have a natural embedding
\begin{equation} \label{E:GammatoCent1}
\Gamma_{M_\ba}(R) \hookrightarrow \Cent_{\GL_\ell(R)} (R
\otimes_{\FF_q(t)} \End(M_\ba)).
\end{equation}

Now since $\bR_f \subseteq \End_{\ok[t,\sigma]}(H(\ba))$, we have
\[
  R_f := \FF_q(t) \otimes_{\FF_q[t]} \bR_f \hookrightarrow \FF_q(t) \otimes_{\FF_q[t]} \End_{\ok[t,\sigma]}(H(\ba)) \cong \End(M_\ba),
\]
where the right-most isomorphism is from
\cite[Prop.~3.4.5]{Papanikolas}.  Therefore, for any
$\FF_q(t)$-algebra $R$, we have an injection
\[
  R \otimes_{\FF_q(t)} R_f \hookrightarrow R \otimes_{\FF_q(t)} \End(M_\ba),
\]
and so from \eqref{E:GammatoCent1} we have
\begin{equation} \label{E:GammatoCent2}
  \Gamma_{M_\ba}(R) \hookrightarrow \Cent_{\GL_\ell(R)} ((R \otimes_{\FF_q(t)} R_f)^\times).
\end{equation}
Let $\Res_{R_f/\FF_q(t)} \Gm$ be the Weil restriction of scalars for
$\Gm_{/ R_{f}}$, and note that for any $\FF_q(t)$-algebra $R$,
\[
  (\Res_{R_f/\FF_q(t)} \Gm)(R) = (R \otimes_{\FF_q(t)} R_f)^\times.
\]
Since $[R_f:\FF_q(t)] = \ell$, we see that $\Res_{R_f/\FF_q(t)} \Gm$
is an $\ell$-dimensional maximal torus inside $\GL_{\ell/\FF_q(t)}$,
and hence
\[
  \Cent_{\GL_\ell}(\Res_{R_f/\FF_q(t)} \Gm ) = \Res_{R_f/\FF_q(t)} \Gm.
\]
{}From \eqref{E:GammatoCent2} we see that $\Gamma_{M_\ba}$ is a
torus.
\end{proof}

Let $E_f = \ok \bigl( \{\tpi\} \cup \{ \Gamma(r) \mid r \in
\frac{1}{f}A \setminus (\{0\} \cup -A_+ ) \} \bigr)$.  We choose a
finite subset $B_f$ of $\cA_f$ whose elements are effective and of
positive weight so that
\begin{equation} \label{E:EfSimplified}
  E_f = \ok \bigl( \cup_{\ba \in B_f} \{ \Pi(a \star \ba)^{-1} \mid a \in A,\ (a,f) = 1\} \bigr).
\end{equation}
For each $\ba \in B_f$, let $\Phi_\ba$ and $\Psi_\ba$ be given as in
Theorem~\ref{T:ABPPhiPsiPi}.  Defining
\[
  M_f := \oplus_{\ba \in B_f} M_\ba,
\]
we see that multiplication by $\sigma$ on $M_f$ is represented by
$\Phi_f  := \oplus_{\ba \in B_f} \Phi_\ba$ and has rigid analytic
trivialization $\Psi_f := \oplus_{\ba \in B_f} \Psi_\ba$.  Since by
Proposition~\ref{P:GammaMaTorus} each Galois group $\Gamma_{M_\ba}$
is a torus, we see that the Galois group $\Gamma_{M_f}$ is also a
torus because by \eqref{E:GammaPsiDefn} we have
\[
  \Gamma_{M_f} \subseteq \prod_{\ba \in B_f} \Gamma_{M_\ba}.
\]
On the other hand, Theorem~\ref{T:ABPPhiPsiPi} and
\eqref{E:EfSimplified} imply that $\ok(\Psi_f(\theta)) = E_f$, and
hence by Theorems~\ref{T:ABPMain} and~\ref{T:PapanikolasMain},
\[
\dim \Gamma_{M_f} = 1 + \frac{q-2}{q-1} \cdot \# (A/f)^\times.
\]

\section{Carlitz $\zeta$-values and $t$-motives}

\subsection{Carlitz polylogarithms and $\zeta$-values}
\label{S:PolyZeta} Let $L_{\alpha,n}(t)$ be the series introduced in
(\ref{E:Lalpha}). We recall the following theorem of Anderson and
Thakur expressing $\zeta_{C}(n)$ in terms of Carlitz polylogarithms.

\begin{theorem}[Anderson-Thakur {\cite[\S 3.7--8]{AndersonThakur}}] \label{T:AndThak}
Given any positive integer $n$, there is a finite sequence
$h^{[n]}_1, \dots, h^{[n]}_{\ell_n} \in k$ with $\ell_n <
\frac{nq}{q-1}$, such that
\begin{equation} \label{E:AndThak}
  \zeta_C(n) = \sum_{i=0}^{\ell_n} h^{[n]}_i \log_C^{[n]}(\theta^i)= \sum_{i=0}^{\ell_n} h^{[n]}_i L_{\theta^i,n}(\theta) .
\end{equation}
\end{theorem}

Let $n$ be a positive integer not divisible by $q-1$.  We let
\[
N_n := \textnormal{$k$-span of $\{\tpi^n, \log_C^{[n]}(1),
\log_C^{[n]}(\theta), \dots, \log_C^{[n]}(\theta^{\ell_n}) \}$.}
\]
By \eqref{E:AndThak}, we have $\zeta_C(n) \in N_n$ and $\dim_k N_n
\geq 2$ since $\zeta_C(n)$ and $\tpi^n$ are linearly independent
over $k$ (as $(q-1)\nmid n$ implies $\tpi^n \notin k_\infty$).  Therefore we can
pick a non-negative integer $m_n$ with $m_n + 2 = \dim_k N_n$ and
distinct integers
\begin{equation} \label{E:indices1}
0 \leq \iota(0), \dots, \iota(m_n) \leq \ell_n
\end{equation}
so that both
\begin{equation} \label{E:indices2}
  \{ \tpi^n, \log_C^{[n]}(\theta^{\iota(0)}), \dots, \log_C^{[n]}(\theta^{\iota(m_n)}) \}, \quad
  \{ \tpi^n, \zeta_C(n), \log_C^{[n]}(\theta^{\iota(1)}), \dots
  \log_C^{[n]}(\theta^{ \iota(m_n)}) \},
\end{equation}
are $k$-bases of $N_n$.  This is possible by
Theorem~\ref{T:AndThak}.

\subsection{The $\zeta$-motive $M_{(s)}$} \label{S:ZetaMotive}
Continue with the choices in \S \ref{S:PolyZeta}.  We consider the
matrices $\Phi_n \in \GL_{m_n + 2}(\ok(t)) \cap\Mat_{m_n +
2}(\ok[t])$,
\[
\Phi_n := \left[\begin{matrix}
(t-\theta)^n & 0  & \cdots & 0 \\
\theta^{\iota(0)/q}(t-\theta)^n & 1 & \cdots & 0 \\
\vdots & \vdots & \ddots & \vdots \\
\theta^{\iota(m_n)/q}(t-\theta)^n & 0 & \cdots & 1
\end{matrix}\right],
\]
and $\Psi_n \in \GL_{m_n+2}(\TT)\cap \Mat_{m_n+2}(\EE)$ given by
\[
\Psi_n  := \left[\begin{matrix}
\Omega^n & 0  & \cdots & 0 \\
L_{\theta^{\iota(0)},n}\Omega^n & 1 & \cdots & 0 \\
\vdots & \vdots & \ddots & \vdots \\
L_{\theta^{\iota(m_n)},n}\Omega^n & 0 & \cdots & 1
\end{matrix}\right].
\]
By \eqref{E:OmegaFnEq} and \eqref{E:LalphanFnEq}, we have
\[
  \Psi_n^{(-1)} = \Phi_n\Psi_n.
\]
Note that all of the entries of $\Psi_n$ are in $\EE$ and that $\Phi_n$
defines a $t$-motive $M_n$ (see \cite[Lem.~A.1]{ChangYu}).

In the context of applying Theorem~\ref{T:PapanikolasMain}, we see
from our various choices arising from~\eqref{E:AndThak} that
$\zeta_C(n)/\tpi^n$ is a $k$-linear combination of the entries of the first
column of $\Psi_n(\theta)$, so in order to compare different
$\zeta$-values simultaneously we need to consider various $M_n$ and
$\Psi_n$ simultaneously.  To do this, we fix a positive integer $s$
and let
\[
  U(s) := \{ n \in \ZZ \mid 1 \leq n \leq s,\ p \nmid n,\ (q-1)\nmid n \}.
\]
We let $M_{(s)}$ be the direct sum $t$-motive $\oplus_{n \in U(s)}
M_n$ and define block diagonal matrices,
\begin{align*}
  \Phi_{(s)} &:= \oplus_{n \in U(s)} \Phi_n, \\
  \Psi_{(s)} &:= \oplus_{n \in U(s)} \Psi_n.
\end{align*}
Then $\Phi_{(s)}$ represents multiplication by $\sigma$ on
$M_{(s)}$, and $\Psi_{(s)}$ is a rigid analytic trivialization of
$\Phi_{(s)}$.  {}From \eqref{E:GammaPsiDefn} we see that any element
of $\Gamma_{\Psi_{(s)}}$ is of the form
\[
  \oplus_{n \in U(s)} \left[\begin{matrix}
  x^n & 0 \\ * & \Id_{m_n+1} \end{matrix}\right],
\]
where $\Id_r$ denotes the identity matrix of size $r$.

Now the Carlitz motive $C$ is a sub-$t$-motive of $M_1 \subseteq
M_{(s)}$, so by Tannakian duality there is a surjective map
\begin{equation} \label{E:piDefn}
  \pi : \Gamma_{\Psi_{(s)}} \twoheadrightarrow \Gm.
\end{equation}
Let $\cT_{M_{(s)}}$ and $\cT_C$ be the strictly full Tannakian
subcategories of the category $\cT$ of $t$-motives, which are
generated by $M_{(s)}$ and $C$ respectively.  The map $\pi$ comes
from the restriction of the fiber functor of $\cT_{M_{(s)}}$ to
$\cT_C$.  For any $\gamma \in \Gamma_{\Psi_{(s)}}(\oFqt)$, it
follows from the discussion in Remark~\ref{R:TautRep} that the
action of $\pi(\gamma)$ on $\oFqt \otimes_{\FF_q(t)} C^{\rB}$ is
equal to the action of the upper left-most corner of $\gamma$.  This
implies that $\pi$ is the projection on the upper left-most corner
of elements of $\Gamma_{\Psi_{(s)}}$.

Let $V_{(s)}$ be the kernel of $\pi$, giving an exact sequence of
linear algebraic groups,
\[
  0 \to V_{(s)} \to \Gamma_{\Psi_{(s)}} \stackrel{\pi}{\to} \Gm \to  1.
\]
We see that $V_{(s)}$ is contained in the vector group $G_{(s)}$,
consisting of all block diagonal matrices of the form,
\[
  \oplus_{n \in U(s)} \left[\begin{matrix}
  1 & 0 \\ * & \Id_{m_n+1} \end{matrix} \right],
\]
which is a unipotent group isomorphic to the direct product
$\prod_{n \in U(s)} \Ga^{m_n+1}$.

\begin{theorem}[Chang-Yu {\cite[Thm.~4.5, Cor.~4.6]{ChangYu}}, cf.\ Theorem~\ref{T:ChangYu}] \label{T:ChangYuMain}
Fix a positive integer $s$, and let $M_{(s)}$, $\Psi_{(s)}$,
$V_{(s)}$, and $G_{(s)}$ be defined as above.  Then $V_{(s)} =
G_{(s)}$, and $\Gamma_{M_{(s)}}$ is an extension of $\Gm$ by a
vector group. Hence
\[
  \dim \Gamma_{\Psi_{(s)}} = 1 + \sum_{n \in U(s)} (m_n+1).
\]
In particular, the quantities among
\[
  \{ \tpi \} \cup_{n \in U(s)} \cup_{i=0}^{m_n} \bigl\{ \log_C^{[n]}\bigl(\theta^{\iota(i)}\bigr) \bigr\}
\]
are algebraically independent over $\ok$.
\end{theorem}

\begin{remark} \label{R:ZetaTransBasis}
When we combine Theorem~\ref{T:ChangYuMain} with the choices we made
in \eqref{E:indices1}--\eqref{E:indices2}, we see that the field
$\ok(\Psi_{(s)}(\theta))$ has a transcendence basis $T\cup \{ \tpi
\}$ with $T\supseteq\{ \zeta_C(n) \mid n \in U(s) \}$.
\end{remark}

\section{Algebraic independence of $\Gamma$-values and $\zeta$-values}
\subsection{The main theorem}
To study the algebraic relations among $\Gamma$-values and
$\zeta$-values simultaneously, we now combine $\Gamma$-motives and
$\zeta$-motives.  Fix a positive integer $s$, and let $M_{(s)}$,
$\Phi_{(s)}$, and $\Psi_{(s)}$ be defined as in
\S\ref{S:ZetaMotive}.

\begin{theorem} \label{T:TransBases}
Assume that $q > 2$.  Suppose $M_0$ is a $t$-motive for which
multiplication by $\sigma$ is represented by $\Phi_0 \in
\Mat_\nu(\ok[t])$ with $\det\Phi_{0}=c(t-\theta)^{m}$, $c\in
\bar{k}^{\times}$, $m \geq 1$, and suppose that $\Psi_0 \in
\GL_\nu(\TT) \cap \Mat_\nu(\EE)$ is a rigid analytic trivialization
of $\Phi_0$. Suppose further that its Galois group $\Gamma_{\Psi_0}$
is a torus over $\FF_q(t)$.  Finally, let $M := M_{(s)} \oplus M_0$,
$\Phi := \Phi_{(s)} \oplus \Phi_0$, and $\Psi :=\Psi_{(s)} \oplus
\Psi_0$.
\begin{enumerate}
\item[(a)] We have $\dim \Gamma_M = \dim \Gamma_{\Psi} = \dim\Gamma_{\Psi_{(s)}} + \dim \Gamma_{\Psi_0} - 1$.
\item[(b)] If $T\cup \{\tpi\}$ is a transcendence basis for $\ok(\Psi_{(s)}(\theta))$ over $\ok$ and $S$ is any transcendence basis of $\ok(\Psi_0(\theta))$ over $\ok$, then the set $T \cup S$ is algebraically independent over~$\ok$.
\end{enumerate}
\end{theorem}

\begin{proof}
By \eqref{E:GammaPsiDefn} we have $\Gamma_{\Psi}\subseteq
\Gamma_{\Psi_{(s)}}\times \Gamma_{\Psi_0}$, so that $\Gamma_{\Psi}$
is a solvable group. By Tannakian duality there is surjective
morphism $\Gamma_\Psi \twoheadrightarrow \Gamma_{\Psi_{(s)}}$, and
it follows that the unipotent radical of $\Gamma_\Psi$ has dimension
$\ge\dim \Gamma_{\Psi_{(s)}}-1$, which is the dimension of the
unipotent radical of $\Gamma_{\Psi_{(s)}}$. On the other hand, we
also have a surjective morphism from $\Gamma_\Psi$ to the torus
$\Gamma_{\Psi_0}$, which forces the dimension of the maximal torus
of $\Gamma_{\Psi}$ to be at least $\dim \Gamma_{\Psi_{0}}$. Hence
$\dim \Gamma_{\Psi} \ge \dim\Gamma_{\Psi_{(s)}} + \dim
\Gamma_{\Psi_0} - 1$.  Observing now that $\det \Psi_0(\theta)$ is a
$\ok^\times$-multiple of $\tpi^{-m}$, we have $\tpi^m \in
\ok(\Psi_0(\theta))$. Since $\tpi$ is also in
$\ok(\Psi_{(s)}(\theta))$, Theorem~\ref{T:PapanikolasMain} gives
part (a).

Part (b) also follows from~(a) by using
Theorem~\ref{T:PapanikolasMain}.  To see this,
\[
  \ok(\Psi(\theta)) \supseteq \ok(T \cup \{\tpi\} \cup S) \supseteq \ok(T \cup \{ \tpi^m \} \cup S) = \ok(T \cup S).
\]
However, these two containments are each of finite degree, and thus
part (a) implies that $\trdeg_{\ok} \ok(T \cup S) = \dim
\Gamma_{\Psi_{(s)}} + \dim \Gamma_{\Psi_0} - 1 = \#T + 1 + \#S - 1 =
\#T + \#S$.
\end{proof}

Our main theorem is as follows.

\begin{theorem} \label{T:MainThmRedux}
Given $f \in A_+$ with positive degree and $s$ a positive integer,
the transcendence degree of the field
\[
  \ok \bigl( \{\tpi\} \cup \bigl\{ \Gamma(r) \bigm| r \in
  \smfrac{1}{f} A \setminus (\{0\} \cup -A_+) \bigr\}
  \cup \{ \zeta_C(1), \dots, \zeta_C(s) \} \bigr)
\]
over $\ok$ is
\[
  1 + \frac{q-2}{q-1}\cdot \#(A/f)^\times + s - \lfloor s/p \rfloor - \lfloor s/(q-1) \rfloor + \lfloor s/(p(q-1)) \rfloor.
\]
\end{theorem}

\begin{proof}
When $q=2$, the result is true, since each $\Gamma(r)$, $r \in k
\setminus A$, is a $\ok$-multiple of $\tpi$  and each $\zeta_C(n)$,
$n \geq 1$, is a $k$-multiple of $\tpi^n$.  When $q > 2$, we can
apply Theorem~\ref{T:TransBases}. Indeed, by
Proposition~\ref{P:GammaMaTorus} we can take $M_0 = M_f$, which
satisfies the hypotheses of Theorem \ref{T:TransBases}, and $S$ to
be a transcendence basis of $E_f$ over $\ok$; by
Remark~\ref{R:ZetaTransBasis}, there is a transcendence basis $T\cup
\{ \tpi \}$ for $\ok(\Psi_{(s)}(\theta))$ over $\ok$ with $T$
containing all $\zeta_C(n)$, $n \in U(s)$.  Since $\# S = 1 +
\frac{q-2}{q-1}\cdot \#(A/f)^\times$ and $\#U(s) = s - \lfloor s/p
\rfloor - \lfloor s/(q-1) \rfloor + \lfloor s/(p(q-1)) \rfloor$, the
result follows.
\end{proof}

\subsection*{Acknowledgement} The authors thank Dinesh Thakur for
many helpful discussions about the contents of this paper.

\bibliographystyle{amsplain}

\end{document}